\documentclass{amsart}
\usepackage{amssymb, amsmath, amsthm}
\usepackage[all]{xy}
\usepackage{cite}
\usepackage{url}

\title{Categories of assemblies for realizability}
\author[W. P. Stekelenburg]{Wouter Pieter Stekelenburg}
\address{Faculty of Mathematics, Informatics, and Mechanics\\
University of Warsaw\\
Banacha 2\\
02-097 Warszawa\\
Poland}
\email{w.p.stekelenburg@gmail.com}

\theoremstyle{plain}
\newtheorem{theorem}{Theorem}
\newtheorem{lemma}[theorem]{Lemma}

\theoremstyle{definition}
\newtheorem{defin}[theorem]{Definition}
\newtheorem{axiom}[theorem]{Axiom}
\newtheorem*{example}{Example}

\newtheorem*{remark}{Remark}

\newcommand\hide[1]{}
\newcommand\cat\mathcal
\newcommand\set[1]{\left\{#1\right\}}
\newcommand\converges{\mathord\downarrow}

\begin{document}
\begin{abstract} This paper introduces \emph{categories of assemblies} which are closely connected to realizability interpretations and which are based on an important subcategory of the effective topos. There is a list of properties which characterize these categories of assemblies up to equivalence. \end{abstract}
\maketitle

\section{Introduction}
The many techniques aggregated in \emph{realizability} demonstrate the consistency of certain classically false statements with non-classical logics, in particular with intuitionistic logic. This paper charts the capabilities of realizability by categorical means. 

This paper is about a class of categories which generalize the \emph{category of assemblies} (see below) which is equivalent to the category of $\neg\neg$-separated objects of the \emph{effective topos}. Each generalized category of assemblies is a Heyting category (theorem \ref{Heyting}). Each category of assemblies implicitly defines a realizability interpretation of its internal language, such that valid and realized propositions coincide (theorem \ref{abracadabra}).
The latter half of this paper outlines some structure and properties which determine when an arbitrary category is equivalent to a category of assemblies (theorem \ref{charac}).

The motivating example of this paper is the category of assemblies. Informally an assembly $a$ is a non-classical subset of an arbitrary set $s$. For each $x\in s$, a set of natural numbers witnesses the membership of $x$ to $a$. These numbers are the \emph{realizers} of $x\in a$. For each subassembly $a'$ of each set $s'$ and for each function $f:s\to s'$, \emph{partial recursive functions} determine whether the restriction of $f$ to $a$ factors through $a'$. Formally, this is defined as follows.

\newcommand\pow{\mathbf P}
\newcommand\N{\mathbb N}
\newcommand\partar\rightharpoonup
\newcommand\dom{\mathrm{dom}}
\newcommand\cod{\mathrm{cod}}
\newcommand\id{\mathrm{id}}
\newcommand\Asm{\mathsf{Asm}}
\begin{defin} An \emph{assembly} is a pair $(X,\alpha)$ where $X$ is a set and $\alpha$ is a function $X\to \pow \N$ valued in nonempty subsets of $\N$. 

Let $(X,\alpha)$ and $(Y,\beta)$ be two assemblies. A \emph{total morphism} $(X,\alpha) \to (Y,\beta)$ is a function $f:X\to Y$ which has a partial recursive $g:\N\partar\N$ such that for each $x\in X$ and $n\in \alpha(x)$, $n$ is in the domain $\dom g$ of $g$ and $g(n)\in \beta(f(x))$.

Assemblies and total morphisms together form the category $\Asm$ of assemblies.\label{egg}
\end{defin}

The category of assemblies is closely related to the \emph{effective topos} of Hyland \cite{MR2479466, MR717245, MR1056382}. On one hand $\Asm$ is equivalent to the subcategory of $\neg\neg$-separated objects of the effective topos and on the other hand the effective topos is the \emph{ex/reg completion} of $\Asm$, i.e.\ the result of freely adding quotients to internal equivalence relations in a way that respects regular epimorphisms \cite{MR1600009, MR1056382, MR1358759, Menni00exactcompletions, MR1900904}. I choose to discuss the category of assemblies rather than its ex/reg completion, because its properties are more stable.

This paper generalizes the category of assemblies along two lines. Firstly, different sets of realizers and computable functions replace the natural numbers and partial recursive functions. Secondly, the object and morphisms of other categories replace sets and functions. The following conditions confine the generalizations this paper presents.
\begin{itemize}
\item The base category is a \emph{Heyting category} (see definition \ref{heytcat}).
\item All notions of computability are Turing complete.
\item There are no \emph{order partial combinatory algebras} \cite{MR1981211, MR2479466} other than partial combinatory algebras in this paper.
\item There is no additional structure on the object of realizers in this paper, as found in modified and extensional realizability \cite{MR1640330, MR0106838, MR1443487, MR2479466}.
\end{itemize}

Weaker than Turing complete notions of computability can still give interesting realizability categories as the following papers show: \cite{MR2265872, Longley_matchingtyped, afsort}. The approaches in these papers are limited by what can either be seen as an extra constraint on the base category (a projective terminal object) or a different constraint on the notion of computability (inhabited families of computable function have at least one computable global section). This paper shows how to overcome these limitations, but only for the special case Turing complete computability. The other constraints are also aimed at keeping things simple, although I have developed a lot of theory for order partial combinatory algebras elsewhere \cite{MSC:8896618, RealCats}.

This paper omits theory on functors between categories of assemblies and to other categories. Much of this can be found in my other work \cite{MSC:8896618, RealCats}. There are straightforward generalizations of Longley's \emph{applicative morphisms} \cite{RTnLS} which are a useful tool for studying regular \ref{regular} functors between realizability toposes. Hofstra and Frey generalize these in \cite{afsort, MR2265872} for other realizabilities. The reason I do not get deeper into these right now, is that there is a notion of applicative morphism related to finite limit preserving functors between realizability toposes (the \emph{left exact morphisms} in \cite{RealCats}) which I am exploring for another paper.

The characteristic properties of realizability categories are formulated as a list of axioms in this paper (axioms \ref{separ}, \ref{weakgen} and \ref{tracking}). They may be translated in the internal language of the categories with some extra effort, to give an axiomatization of realizability like Dragalin's \cite{MR0258596} and Troelstra's \cite{MR0335240}. I have outlined the possibilities in my thesis \cite{RealCats}, but this extra effort will have to wait for another paper too.

I have avoided realizability triposes and similar structures throughout this paper. This should make the content accessible to readers who are unfamiliar with categorical realizability. However, familiarity with categorical logic is \emph{not} a luxury.

\section{Ingredients}\label{Ingredients}
The generalized categories of assemblies look like this:
the objects and morphisms of an arbitrary \emph{Heyting category} replace set and functions;
some object $A$ replaces $\N$;
partial morphisms $A\partar A$ which are \emph{$\phi$-computable} relative to a partial \emph{application} operator $\cdot:A\times A\partar A$ and a \emph{combinatory complete filter} $\phi$, replace the partial recursive functions.
This section defines what the emphasized words in the previous sentence mean.

\subsection{Heyting categories}
Heyting categories correspond to theories in many sor\-ted first order intuitionistic logic. Theories induce Heyting categories of definable functions and Heyting categories have an internal language which is a many-sorted first order intuitionistic logic. 

\newcommand\sub{\mathsf{Sub}}
\newcommand\pre[1]{#1^{-1}}
\newcommand\im[1]{\exists_{#1}}
\begin{defin}[Heyting category]\label{heytcat} Let $\cat C$ be any category. For each object $X$ of $\cat C$, a \emph{subobject} is an isomorphism class of monomorphisms into $X$. Factorization induces a partial order $\subseteq$ on subobjects. The poset of all subobject is $\sub(X)$. If $\cat C$ has all pullbacks, then for each morphism $f: X\to Y$ pulling back monics induces a function $f^{-1}:\sub(Y)\to\sub(X)$ called the \emph{inverse image map}.

A \emph{Heyting category} is a category where the poset of subobjects $\sub(X)$ is a Heyting algebra for each object $X$ and where the inverse image map of each arrow is a homomorphism of Heyting algebras which has both adjoints. This means that
\begin{itemize}
\item for each object $X$, $\sub(X)$ has finitary joins ($\emptyset, \cup$), finitary meets ($X, \cap$) that distribute over joins and a map $d: \sub(X)\times \sub(X)\to \sub(X)$ that satisfies $x\cap y\subseteq z$ if and only if $x\subseteq d(y,z)$; 
\item for each morphism $f:X\to Y$, $f^{-1}$ preserves $d$ and there are $\forall_f,\exists_f:\sub(X)\to\sub(Y)$ such that $x\subseteq \forall_f(y)$ if and only if $f^{-1}(x)\subseteq y$ and $\exists_f(x)\subseteq y$ if and only if $x\subseteq f^{-1}(y)$; 
\item the adjoints satisfy the \emph{Beck-Chevalley} condition: if $f\circ g'=g\circ f'$ is a pullback square, then $\im{g'}\circ \pre{(f')} = \pre{f}\circ \im{g}$ and $\forall_{g'}\circ \pre{(f')} = \pre{f}\circ \forall_{g}$.
\[ \xymatrix{
\bullet \ar[r]^{f'}\ar[d]_{g'}\ar@{}[dr]|<\lrcorner & \bullet \ar[d]^g\ar@{}[drr]|\Longrightarrow && \bullet \ar[d]_{\im{g'}} & \bullet \ar[l]_{\pre{(f')}}\ar[d]^{\im g}\\
\bullet \ar[r]_f & \bullet && \bullet & \bullet \ar[l]^{\pre f}
}\]
\end{itemize}
\end{defin}

\begin{example} Every topos, including the category of sets and the effective topos is a Heyting category. In fact all regular (see definition \ref{regular}) extensive locally Cartesian closed categories, like the category of assemblies from the introduction, are Heyting. \end{example}

Heyting algebras admit an interpretation of a propositional language for which intuitionistic propositional logic is sound. The rest of the structure handles the extension to first order logic. The inverse image map takes care of substitutions. The adjoints take care of quantification and equality. The Beck-Chevalley condition says that quantification commutes with substitution.

\newcommand\of{\mathord:}
\newcommand\subob[1]{ {\left\langle #1 \right\rangle} }
\begin{defin} First order logic defines subobjects of objects in Heyting categories. I regularly exploit this in the definitions below. The following notation sets the formulas and subobjects of the internal language apart form the formulas and sets of the metalanguage:
\begin{itemize}
\item For $u\in \sub(X)$, is the related predicate is denoted $x\of u$.
\item For each formula $\chi$ with no other free variables than $x$ of type $X$, $\subob{x\of X| \chi(x)}$ defines the related subobject of $X$.
\end{itemize}
\end{defin}

\begin{defin}[Representative monomorphisms] \label{promonos} There is little harm in only considering small categories, because by G\"odels completeness theorem there are small models of set theory. This makes $\sub(X)$ a large set of large sets instead of a proper class of proper classes. The advantage is that for each small category with finite limits, there is a function that chooses for each object $X$ and each subobject $U\in \sub(X)$ a monomorphism $\mu_U:X_U \to X$ by the axiom of choice. Such representative monomorphisms figure in proofs throughout this paper.
\end{defin}

\subsection{Partial combinatory algebras}
The starting point for the model of computability is the universal computer: this is a device that takes a string of code and an string of input, performs a computation and  produces a string of output provided that the computation terminates. The following structure captures this idea inside a Heyting category.

\begin{defin}[Partial applicative structure] Let $\cat B$ be a Heyting category. A \emph{partial applicative structure} is a pair $(A,\cdot)$ where $A$ is an object and $\cdot$ is a partial morphism $A\times A\partar A$. Here, a \emph{partial morphism} $X\partar Y$ is an isomorphism class of spans $(m:Z\to X, f:Z\to Y)$ in which $m$ is monic. I refer to $\cdot$ as the \emph{application operator}. \end{defin}

The desired model of computation allows computable functions to work on non-computable data and therefore permits subobjects of $A$ to exclude computable elements. The following structure `filters' those non-computable subobjects out.

\begin{defin}[Filter] Let $\cat B$ be a Heyting category and let $A$ be a partial applicative structure in $\cat B$. A \emph{filter} is a \emph{subset} $\phi\subseteq\sub(A)$ with the following properties.
\begin{itemize}
\item If $U\in \phi$, then $U$ is \emph{inhabited}, i.e.\ the canonical map form $A_U$ to the terminal object $1$ is a regular epimorphism.
\item If $U\in \phi$ and $U\subseteq V$, then $V\in \phi$.
\item If $U,V\in \phi$ and $U\times V\subseteq \dom(\cdot)$ -- the domain of the application operator -- then $\im{(\cdot)}(U\times V) \in \phi$. Here $U\times V$ is actually $\pi_0^{-1}(U)\cap\pi_1^{-1}(V)$ where $\pi_0,\pi_1:A\times A\to A$ are the projection maps.
\end{itemize}
\end{defin}

\begin{example} The set of inhabited subobjects is a filter. If $C\subseteq A$ is a subobject such that for all $x,y\in C$, if $(x,y)\of\dom(\cdot)$ then $x\cdot y \in C$, then the set of subobjects that intersect $C$ -- i.e.\ $U\cap C$ is inhabited -- is a filter. \end{example}

Computability means `having a computable code on the universal machine' and is therefore formalized as follows.

\begin{defin}[Computable] Let $\cat B$ be a Heyting category, let $A$ be a partial applicative structure in $\cat B$ and let $\phi$ be a filter on $A$. For each $n\in\N$ an $n$-ary partial morphism $f:A^n\partar A$ is \emph{$\phi$-computable} if there is an $R\in \phi$ such that
\begin{itemize}
\item for all $r\in F$ and $\vec x\in A^{n-1}$, $((r\cdot x_1)\dotsm x_{n-1})$ is defined
\item for all $r\in F$ and $\vec y\in \dom f$, $((r\cdot y_1)\dotsm y_n)$ is defined and equal to $f(\vec y)$
\end{itemize}
Both of these statements should be interpreted in the internal language of $\cat B$. Therefore let $a_1:A\to A$ be the identity and let $a_{n+1}:A^{n+1} \partar A$ for $n>0$ satisfy:
\[ \dom(a_{n+1}) = (a_n\times \id)^{-1}(\dom(\cdot)),\quad a_{n+1} = (\cdot) \circ (a_n\times \id) \]
The reformulation is:
\begin{itemize}
\item $F\times A^{n+1} \subseteq \dom(a_n)$
\item $F\times \dom f\subseteq\dom(a_{n+1})$ and $a_{n+1}$ restricts to $f\circ\pi_{1\dotsc n}$ where $\pi_{1\dotsc n}$ is the projection $F\times \dom f\to\dom f$
\end{itemize}\label{computable}
\end{defin}

The following condition ensures that $\phi$-computability has enough power to realize the soundness of first order logic.

\begin{defin}[Combinatory completeness] Let $\cat B$ be a Heyting category, let $A$ be a partial applicative structure in $\cat B$ and let $\phi$ be a filter on $A$. The filter $\phi$ is combinatory complete if all projections $\vec x\mapsto x_i:A^n \to A$ are $\phi$-computable and if $\phi$-computable maps are closed under pointwise applications, i.e.\ if $f:A^n\partar A$ and $g:A^n\partar A$ are $\phi$-computable, then $(\cdot)\circ (f,g):A^n \partar A$ defined on $(f,g)^{-1}(\dom(\cdot))$ is $\phi$-computable too.
\end{defin}

\begin{example} There is a universal partial recursive function $\cdot:\N\times\N\partar \N$. The filter of inhabited subobjects is combinatory complete for the applicative structure $(\N,\cdot)$. This holds for each natural number object in each Heyting category. \end{example}

The following structures connect applicative structures to \emph{combinatory logic} \cite{MR1512218}.

\newcommand\comb\mathbf
\begin{defin}[Partial combinatory algebras] A \emph{partial combinatory algebra} is a partial applicative structure $(A,\cdot)$ with inhabited $\comb k,\comb s\in \sub(A)$ that satisfy:
\begin{itemize}
\item for all $k\in \comb k$ and $x,y\in A$, $(k\cdot x)\cdot y = x$.
\item for all $s\in \comb s$ and $x,y\in A$, $(s\cdot x)\cdot y$ is defined.
\item for all $s\in \comb s$ and $x,y,z\in A$, if $(x\cdot z)\cdot(y\cdot z)$ is defined then $((s\cdot x)\cdot y)\cdot z = (x\cdot z)\cdot(y\cdot z)$.
\end{itemize}
\end{defin}

\begin{theorem} A partial applicative structure $(A,\cdot)$ is a partial combinatory algebra if and only if it has a combinatory complete filter. \end{theorem}

\begin{proof} Let $\phi$ be a combinatory complete filter on $(A,\cdot)$. The morphisms $(x,y)\mapsto x$ and $(x,y,z)\mapsto (x\cdot z)\cdot (y\cdot z)$ are computable by combinatory completeness. Hence there are suitable $\comb k$ and $\comb s\in \phi$. Since members of $\phi$ are inhabited $(A,\cdot)$ is a partial combinatory algebra.

Converse: the filter of inhabited subobjects of $A$ is combinatory complete, because the combinators $k$ and $s$ form a complete basis for combinatory logic (see \cite{MR1512218}).
\end{proof}

\begin{remark} The definition of computability says that as long as some parts of the input of a computable function are missing, the universal computer makes a computation which always halts. This is realistic, as a real computer can just store input and only start computing when every bit of input is collected. Without this property of computability, the combinatory complete filters are connected to \emph{conditional partial combinatory algebras} \cite{MR1233148}, which are equivalent to ordinary partial combinatory algebras for realizability (by proposition 1.2.3 of \cite{MR2479466}), but harder to use.
\end{remark}

\subsection{Computability}
This subsection contains technical results for the proof that the generalized categories of assemblies are Heyting categories. That proof is essentially the soundness theorem of first order intuitionistic logic for realizability interpretations. This section defines certain relations on partial applicative structure with the help of the \emph{$\lambda$-calculus} \cite{MR622912}, and demonstrates that these $\lambda$-definable relations contain $\phi$-computable morphisms for any filter $\phi$.

The $\lambda$-calculus describes a set of functions that act on each other. They determine a notion of computability that is equivalent to Turing computability.

\newcommand\la{$\lambda$}
\newcommand\cto{\mathrel\to_{\beta\eta}}
\begin{defin} A \emph{\la-term} $M$ is either a \emph{variable symbol} $x,y,z,\dots$ from some infinite \emph{set} of variables $V$, an \emph{application} $MN$ of \la-terms $M$ and $N$, or an \emph{abstraction} $(\lambda x.M)$ where $x$ is a variable symbol and $M$ is a $\lambda$-term. The set of all \la-terms is $\Lambda$. \emph{Substitution} is the following operation on \la-terms. For all variable symbols $x,y$ with $x\neq y$ and all \la-terms $M,N,P$:
\begin{align*}
(\lambda x.M)[P/y] &= \lambda x.(M[P/y]) & x[P/y] &= x & (MN)[P/y] &= (M[P/y])(N[P/y]) \\
(\lambda y.M)[P/y] &= \lambda y.M & y[P/y] &= P
\end{align*}
The \emph{$\beta\eta$-conversion} relation is the least preorder $\cto$ on \la-terms that satisfies:
\begin{itemize}
\item stability: if $M\cto N$, then $M[P/x]\cto N[P/x]$;
\item costability: if $M\cto N$, then $P[M/x]\cto P[N/x]$;
\item adjunction: $Mx\cto N$ if and only if $M\cto \lambda x.N$.
\end{itemize}
\end{defin}

\begin{remark} This definition highlights properties of $\beta\eta$-conversion that are convenient in this context. Ordinarily, $\cto$ is defined as the least preorder that satisfies:
\begin{itemize}
\item $\alpha$-equivalence: $\lambda x.M \cto \lambda y.M[y/x]$;
\item $\beta$-reduction: $(\lambda x.M)N\cto M[N/x]$;
\item $\eta$-expansion: $M\cto \lambda x.(Mx)$;
\item head reduction: if $M\cto N$, then $MP\cto NP$;
\item tail reduction: if $M\cto N$, then $PM\cto PN$;
\item $\zeta$-rule: if $M\cto N$, then $\lambda x.M\cto \lambda x.N$.
\end{itemize}
The proof that these conditions define the same relation is an exercise for the reader. See \cite{MR622912} for more \la-calculus.
\end{remark}

Any partially ordered set $P$ with monotone maps $\alpha: P\to P^P$ and $\lambda: P^P \to P$ such that $\alpha\dashv \lambda$ allows and interpretation of \la-terms which respects $\beta\eta$-conversion. Here the fibred poset $\sub(-\times A)$ plays the role of $P$ in the following manner.
The sets $\sub(X\times A\times A)$ represent certain functions $\sub(X\times A) \to\sub (X\times A)$, namely functions of the following form.
\[ V_*(U) = \subob{(x,b)\of X\times A| \exists (x,a)\of U. (x,a,b)\of V} \textrm{ for } V\in \sub(X\times A)\]
Here (and from now on) $\exists (x,a)\of U.\chi$ is short for $\exists (x',a)\of U. x=x'\land\chi$.
An adjunction between $\sub(A)$ and $\sub(A\times A)$ determines a lax kind of model of the $\lambda$-calculus.
Informally, the adjunction comes from the inverse and direct image maps of the `function':
\[ a,b\mapsto \subob{c\of A| (a,b)\of\dom(\cdot) \to a\cdot b = c }\]
The following formalizes this in a Heyting category without power objects.

\begin{defin} Form now on, let $x\cdot y\converges$ stand for $(x,y)\of\dom(\cdot)$.
For each object $X$ of $\cat B$ and each $U\in\sub(X\times A)$ and $V\in \sub(X\times A\times A)$ let
\begin{align*}
\alpha_X(U) &= \subob{(x,c,d)\of A^n\times A \times A| \exists (x,b)\of U. b\cdot c\converges \to b\cdot c = d }\\
\lambda_X(V) &= \subob{(x,b)\of A^n\times A| \forall c\of A. b\cdot c\converges\to (x,c,b\cdot c)\of V }\\
\end{align*}
\end{defin}

\begin{lemma}[$\alpha\dashv \lambda$] For all $U\in \sub(X\times A)$ and $V\in \sub(X\times A\times A)$, $\alpha_X(U)\subseteq V$ if and only if $U\subseteq \lambda_X(V)$. \label{AAA}\end{lemma}

\begin{proof} Assume $\alpha_X(U)\subseteq V$, i.e.\ for all $c,d\of A$, if there is an $(x,b)\of U$ such that $b\cdot c = d$ if $b\cdot c$ is defined, then $(x,c,d)\of V$. Let $(x,b)\of U$ and let $c\of A$ be such that $b\cdot c$ is defined. Then $(x,c,b\cdot c)\of V$ by assumption and $U\subseteq \lambda_X(V)$ follows by generalization. Now assume $U\subseteq \lambda_X(V)$, i.e.\ for all $(x,b)\of U$ and $c\of A$ such that if $b\cdot c$ is defined, then $(x,c,b\cdot c)\of V$. Let $(x,c,d)\of \alpha(U)$, so there is an $(x,b)\of U$ such that $b\cdot c = d$ if defined. Then $(x,c,b\cdot c) = (x,c,d)\of V$ by assumption and $\alpha_X(U)\subseteq V$ follows by generalization.
\end{proof}

Let me show how the fibred adjunction $\alpha_X\dashv \lambda_X$ satisfies $\cto$.

\newcommand\db[1]{\left[\!\left[ #1 \right]\!\right]}
\newcommand\FV{\mathrm{FV}}
\begin{defin} Let $V$ be enumerated: $V = \set{x_0,x_1,x_2,\dots}$. For each \la-term $M$ the set of \emph{free variables} $\FV(M)$ is defined as follows:
\[ \FV(x) = x, \quad \FV(MN) = \FV(M)\cup\FV(N), \quad \FV(\lambda x.M) = \FV(M)-\set{x} \]
For each \la-term $M$ let $\# M$ be the greatest $i$ such that $x_i\in \FV(M)$, or $-1$ if $\FV(M) = \emptyset$. For each $n\in \N$ interpret the \la-terms $M$ with $\# M<n$ as follows:
\begin{align*}
\db{x_i}_n &= \subob{(\vec a,a_i)\of A^n\times A| \vec a\of A^n } \\
\db{MN}_n &= \alpha_{A^n}(\db M_n)_*(\db N_n) \\
 &= \subob{(\vec a,d)\of A^n\times A| \exists (\vec a,b)\of \db M_n,(\vec a,b)\of \db N_n. b\cdot c\converges \to b\cdot c = d} \\
\db{\lambda x_i.M}_n &= \lambda_{A^n}(\db{M[x_n/x_i]}_{n+1})
\end{align*}
\end{defin}

\begin{lemma} For each $n\in \N$ and each pair of \la-terms $M,N$ such that $\# M$, $\# N\leq n$, if $M\cto N$, then $\db M\subseteq \db N$. \end{lemma}

\begin{proof} This is a proof by nested induction, namely induction over the rules that define substitution inside induction over the rules that define conversion. 

The following equality holds by induction over the substitution rules:
\[ \db{M[P/x_{n+1}]}_n = (\db{M}_{n+1})_*(\db{P}_n) \]
The operation $V_*(U)$ is monotone in both variables and therefore $\db-_n$ respects both stability and costability.
Lemma \ref{AAA} show that $\db-_n$ also respects the adjunction between abstraction and application. Hence if $M\cto N$, then $\db M\subseteq \db N$.
\end{proof}

I extend the possibilities for defining relations with the \la-calculus a little bit further. The lemmas above generalize to the extended language without trouble.

\begin{defin} A \la-term with constants is a \la-term which besides variables has elements of $\sub(A)$ as atomic subterms. The set of \la-terms with constants is $\Lambda_{\sub(A)}$. For $U\in \sub A$, let $\FV(U) = \emptyset$ and for each $n\in \N$ let $\db U_n = A^n\times U$.
\end{defin}

To simplify notation, I use to following conventions: 
\begin{itemize}
\item Leave out subscript $0$, so $\db M = \db M_0$;
\item Leave out repetitions of $\lambda$, as with quantifiers: \[ \lambda x_0,\dotsc, x_n.M = \lambda x_0.\dotsc\lambda x_1. M\]
\item Application has priority over abstraction and associates to the right, so $\lambda x,y,z.xyz = \lambda x,y,z.((xy)z)$.
\end{itemize}


The end of this subsection is about \la-definable subobjects of $A$ that are members of any combinatory complete filter $\phi$. 
Let me explain why not all \la-definable subobjects are elements of every filter. If $U,V\in \phi$, then $UV$ is a \la-term with constants and 
\[ \db {UV} = \subob{w\of A| \exists u\of U,v\of V. u\cdot v\converges \to u\cdot v = w }\]
Classically, $\db {UV}\in \phi$, because either $u\cdot v\converges$ for all $u\of U$ and $v\of V$, in which case $\db {UV} = \im{(\cdot)}(U\times V)\in \phi$ by definition, or $\db {UV} = A$, which is in $\phi$ simply because there are computable morphisms and $\phi$ is upward closed. It is unclear what happens when $\forall u\of U,v\of V.u\cdot v\converges$ is undecidable. Therefore $\phi$ may exclude $\db {UV}$. It turns out that combinatory complete filters only bar applications, however.

\begin{lemma}[Computable terms] Let $\Lambda_\phi$ be \la-terms whose constants are in $\phi$. Let $M\in \Lambda_\phi$ such that $\FV(M)=\emptyset$ and such that $M$ is not the application of two other \la-terms. Then $\db M\in \phi$. \label{computableterms} \end{lemma}

\begin{proof} Since $M$ has no free variables and is no application, it is either a constant in $\phi$ or an abstraction. If $M\in \phi$, then $\db M = M\in \phi$ by assumption. This leaves the case that $N$ is an abstraction. If $M$ is an abstraction, then is it of the form $\lambda x_0,\dotsm, x_n.N$ (for arbitrary $n$) where either $N=x_i$ for some $i\leq n$, $N\in \phi$  or $N = PQ$ for some other \la-terms $P$, $Q$.
\begin{itemize}
\item if $N = x_i$ then $\db{\lambda \vec x.N}$ is precisely the set of codes of the projection $\vec x\mapsto x_i$, which is in $\phi$, because this projection is $\phi$-computable.
\item if $N\in \phi$, then $\db{\lambda y,\vec x.y}\times N\subseteq \dom(\cdot)$ since projections are $\phi$-computable and $\db{\lambda \vec x.N} = \im{(\cdot)}(\db{\lambda y,\vec x.y}\times N)$.
\item if $N=PQ$, $\db{\lambda \vec x. P}\in \phi$ and $\db{\lambda \vec x. Q}\phi$, then $\db{\lambda \vec x. PQ}\in\phi$ for the following reason. Repeated application determines the following families of computable functions:
 \begin{align*}
 a_n&: \db{\lambda \vec x. P}\times A^n \partar A\\
 a_n&: \db{\lambda \vec x. Q}\times A^n \partar A\\
 a_n&: \db{\lambda \vec x. PQ}\times A^n \partar A\\
 \end{align*}
The last family is the pointwise application of the first two. The members of first two are $\phi$-computable because $\db{\lambda \vec x. P}$ and $\db{\lambda \vec x. Q}\in \phi$. Therefore the members the last family $\phi$ computable. Hence $\db{\lambda \vec x. PQ}\in \phi$.
\end{itemize}
So for every $M\in \Lambda_\phi$ such that $\FV(M) = \emptyset$ and $M$ is not the application of two other terms $\db M\in \phi$ by induction over subterms of $M$.
\end{proof}

\section{Categories of assemblies}\label{Catsofasms}
This section generalizes the definition of the category of assemblies in the introduction, replacing sets, functions and partial recursiveness with objects and morphisms from a Heyting categories and partial computable morphisms. I have to make two adjustments. Firstly, assemblies are bundles of subobjects of $A$, because the power object $\pow A$ may be absent from $\cat B$. Secondly, inhabited families of computable morphisms rather than individual morphisms determine which morphisms between assemblies are total.

\newcommand\uo{\mathrm U}
\newcommand\rr[1]{\rho_{#1}}
\begin{defin}[Category of assemblies] Let $\cat B$ be a Heyting category, let $(A,\cdot)$ be a partial combinatory algebra of $\cat B$ and let $\phi$ be a combinatory complete filter of $(A,\cdot)$. An \emph{$A$-assembly} is a pair $(X,\xi)$ where $X$ is an object of $\cat B$ and $\xi$ is a member of $\sub(A\times X)$ that satisfies: \[ \cat B\models \forall x\of X.\exists a\of A. \xi(a,x) \]
For each assembly $X = (Y,\Phi)$ let $\uo X = Y$ and $\rr X = \Phi$ in order to save variables. 

Let $X$ and $Y$ be assemblies. A subobject $R\in \sub(A)$ \emph{tracks} $f:\uo X\to \uo Y$ if
\[ \cat B\models \forall a,b\of A, x\of X. a\of R \land (b,x)\of \rr X \to a\cdot b\converges\land (a\cdot b,y)\of\rr Y \]
A morphism $f:\uo X\to \uo Y$ is \emph{$\phi$-total} if some $R\in\phi$ tracks it.
The $\phi$-total morphisms are closed under composition and identities because computable functions are.
The \emph{category of assemblies} $\Asm(\cat B, (A,\cdot),\phi)$ consists of $A$-assemblies and $\phi$-total morphisms.\label{Asmbaphi}
\end{defin}

An $A$-assembly $X$ is a fresh subobject of the object $\uo X$ in $\cat B$ whose membership is determined by the relation $\rr X$. A morphism between underlying object is $\phi$-total, if there are $\phi$-computable morphisms to take care of the realizers. This is all parallel to how the ordinary category of assemblies was defined. In fact, that category is an example.

\begin{example} If $\cat S$ is the category of sets, $\cdot:\N\times\N\partar \N$ a universal partial recursive function and if $\phi$ is the filter of inhabited subobjects of $\N$, then $\Asm(\cat S,(\N,\cdot), \phi)$ and the category of assemblies of definition \ref{egg} are equivalent. \end{example}

\begin{remark}[Alternative] Spans $(a:Y\to A, x:Y\to X)$ represent every subobject in $\sub(A\times X)$ at least once. For this reason, there is an equivalent category whose objects are spans and whose definition is less dependent on an ambient set theory or on the internal language of $\cat B$. Such a definition makes the proof that the category of assemblies is a Heyting category much more complicated, however. \end{remark}

\subsection{Regularity}
The rest of this section is a tedious check that categories of assemblies are Heyting categories. This subsection goes halfway with a proof that $\Asm(\cat B,(A,\cdot),\phi)$ is \emph{regular}.


\begin{defin}[Regular] Let $\cat C$ be a category with finite limits. A \emph{kernel pair} of a morphism $f:X\to Y$ in $\cat C$ is a pair of morphism $p,q:W\to X$ such that $f\circ p = f\circ q$ is a pullback square.
\[\xymatrix{
W\ar[r]^p \ar[d]_q \ar@{}[dr]|<\lrcorner & X\ar[d]^f\\
X \ar[r]_{f} & Z
}\]
A regular epimorphism is a morphism that is the coequalizer of its own kernel pair. A $\cat C$ is \emph{regular} if regular epimorphisms are stable under pullback and if each kernel pair has a coequalizer.

A functor between regular categories is regular if it preserves finite limits and regular epimorphisms. \label{regular}
\end{defin}

\begin{example} Categories of algebras and homomorphisms are usually regular. So are all toposes. \end{example}

Upcoming proofs require the following properties of regular categories and functors.

\begin{defin} For each $f:X\to Y$ in a regular category, and each monic $m:W\to X$, $f\circ m$ factors into a regular epimorphism followed by a monomorphism in an up to isomorphism unique way. Since this factorization preserves isomorphisms between monomorphisms, it induces a function $\sub(X)\to\sub(Y)$, which equals the \emph{direct image map} in Heyting categories. Hence I denote it by $\im f$. \end{defin}

\begin{lemma}\label{reglog} Let $\cat C$ be a regular category. For each $f:X\to Y$ of $\cat C$ the inverse image map $\pre f:\sub(Y)\to\sub(X)$ has a left adjoint $\im f$ which satisfies the Beck-Chevalley condition. Each regular $F:\cat \cat C \to \cat D$ be a regular functor preserves monomorphisms and therefore induced family of morphisms $F_X:\sub(X)\to\sub(FX)$, which commutes with all inverse image maps and all of their left adjoints.
\[\xymatrix{
\sub(X)\ar[d]_{F_X} \ar[r]^{\im f} & \sub(Y) \ar[d]_{F_Y} \ar[r]^{\pre f} & \sub(X) \ar[d]^F \\
\sub(FX)\ar[r]_{\im{Ff}} & \sub(FY) \ar[r]_{\pre{Ff}} & \sub(X) 
}\]\end{lemma}

\begin{proof} This proof is left as exercise for the reader. \end{proof}

\begin{defin} The following operators construct assemblies with useful properties.\label{hulpfuncties}
\begin{itemize}
\item For every object $X$ of $\cat B$, $\nabla X = (X,A\times X)$.

Assemblies of this form have the property that each morphism $\uo Y \to X$ determines a $\phi$-total morphism $Y\to \nabla X$. In particular $\id_{\uo Y}:Y\to\nabla\uo Y$ is $\phi$-total for each assembly $Y$.

\item For every assembly $Y$ and every $f:X\to \uo Y$ let:
\[  f^*(Y) = (X,\pre{(\id_A\times f)}(\rr Y)) \]

The important property of this construction is that $f\circ g:W\to Y$ is $\phi$-total if and only if $g:W\to f^*(Y)$ is. The reason is that if $R\in \phi$ tracks one, it automatically also tracks the other.

\item For every assembly $X$ and every regular epimorphism $e:\uo X\to Y$ let:
\[ e_*(X) = (Y, \im{(\id_A\times f)}(\rr X))\]

This construction is dual to the last: its property is that $f\circ e:X\to Z$ is $\phi$-total if and only if $f:e_*(X)\to Z$ is. The reason that this only works for regular epimorphisms, is that each assembly $Y$ has to satisfy $\cat B\models\forall x\of \uo Y. \exists a\of A.\rr Y$. Note that $e_*e^*(Y) = Y$: because $\id_A\times e$ is a regular epimorphism, $\im{(\id_A\times e)}\circ \pre{(\id_A\times e)} = \id_{\sub(\uo Y)}$. Also note that if $f\circ e' = e\circ f'$ is a pullback square in $\cat B$, then $e'_*\circ (f')^* = f^* \circ e_*$ because of the Beck-Chevalley condition in $\cat B$.
\[ \xymatrix{
\bullet \ar[r]^{f'}\ar[d]_{e'}\ar@{}[dr]|<\lrcorner & \bullet \ar[d]^e\ar@{}[drr]|\Longrightarrow && \bullet \ar[d]_{e'_*} & \bullet \ar[l]_{(f')^*}\ar[d]^{e_*}\\
\bullet \ar[r]_f & \bullet && \bullet & \bullet \ar[l]^{f^*}
}\]

\item For every pair of assemblies $Y$ and $Y'$ for which $\uo Y = \uo Y'= X$ let:
\[ Y\otimes_X Y' = \left(X,\subob{(a,x)\of A\times X \middle | \begin{array}{l}
 \forall k\of \db{\lambda x,y.x}.a\cdot k\converges, (a\cdot k,x)\of\rr Y,\\
 \forall l\of \db{\lambda x,y.y}.a\cdot l\converges,(a\cdot l,x)\of\rr{Y'}
\end{array}}\right) \]

In this case, $Y\otimes_X Y'$ is the pullback of the canonical morphisms $\id_X:Y\to\nabla\uo Y$ and $\id_X:Y'\to\nabla\uo Y'$ for the following reasons. The morphisms $\id_X:Y\otimes_X Y'\to Y$ and $\id_X:Y\otimes_X Y'\to Y$ are $\phi$-total because $\db{\lambda x.x(\lambda y,z.y)}$ and $\db{\lambda x.x(\lambda y,z.z)}\in \phi$. Of each pair of $\phi$-total morphisms $p:Z\to Y$ and $q:Z\to Y'$ such that $\id_{\uo Y}\circ p = \id_{\uo Y}\circ q$ the underlying maps are equal. Finally, if $U\in \phi$ tracks $p:Z\to Y$ and $V\in \phi$ tracks $p:Z\to Y'$, then $\db{\lambda xy.y(Ux)(Vx)}$ tracks $p: Z\to Y\otimes_X Y'$. 

Note that for each $f:W\to X$, $f^*(Y\otimes_X Y') = f^*(Y)\otimes_W f^*(X)$.
\end{itemize}
\end{defin}

\begin{remark} It may please readers familiar with fibred categories to know that the underlying object map is a fibred bounded meet semilattice. It has supine morphisms over all the regular epimorphisms and because these supine morphisms are prone, $\uo$ is also a stack for the regular topology of $\cat B$ (see \cite{MR2223406} for more stacks).\end{remark}

\begin{theorem} The category $\Asm(\cat B,(A,\cdot),\phi)$ is regular. Moreover, $\uo$ and $\nabla$ extend into a pair of regular functors $\uo:\Asm(\cat B,(A,\cdot),\phi)\to\cat B$ and $\nabla:\cat B \to \Asm(\cat B,(A,\cdot),\phi)$ where $\uo$ is left adjoint to $\nabla$. \label{basis}
\end{theorem}

\begin{proof} That $\uo$ and $\nabla$ are an adjoint pair of functors is easy to see. Any $\phi$-total morphism $X\to Y$ is a morphism $\uo X\to \uo Y$, any morphism $\uo X\to Z$ is a $\phi$-total morphism $X\to \nabla Z$, so $\cat B(\uo X,Z) = \Asm(\cat B,(A,\cdot),\phi)(X,\nabla Z)$. Because $\uo \nabla W = W$, every morphism $W\to Z$ is $\phi$-total $\nabla W\to \nabla Z$. 

If $1$ is terminal in $\cat B$, then $\nabla 1$ is in $\Asm(\cat B,(A,\cdot),\phi)$ because $\nabla$ preserves limits. For other finite limits let $f:X\to Z$ and $g:Y\to Z$ in $\Asm(\cat B,(A,\cdot),\phi)$. Suppose that $f':W \to \uo Y$ and $g':W\to \uo X$ are pullbacks of $f$ and $g$ along each other. Now let $W'= (f')^*(Y)\otimes_W (g')^*(X)$. For any pair $h:V\to X$ and $k: V\to Y$ such that $f\circ h = g\circ k$, there is a unique $l:\uo V\to W$ such that $h=g'\circ l$ and $k=f'\circ l$ in $\cat B$. This $l$ is a $\phi$-total morphism $V\to W'$, because it is $\phi$-total as a morphism $V\to(f')^*(Y)$ and as a morphism $V\to(g')^*(X)$ and because $W'$ is a pullback of $(f')^*(Y)\to \nabla W$ and $(g')^*(X)\to \nabla W$.

\[\xymatrix{
& V \ar@/^/[drr]^{h}\ar@/^/[ddr]^{k} \ar@{.>}[d]^l\ar@{.>}@/_/[ddl]_l \ar@{.>}@/_/[dl]_l \\
W' \ar[r]\ar[d]\ar@{}[dr]|<\lrcorner & (g')^*(X)\ar[d] \ar[rr]_{g'} && X\ar[d]^f \\ 
(f')^*(Y) \ar@/_/[rr]_{f'} \ar[r] & \nabla W & Y\ar[r]_g & Z 
}\]

This demonstrates that  $\Asm(\cat B,(A,\cdot),\phi)$ has finite limits. The functor $\nabla$ preserves these limits because it is a right adjoint to $\uo$. The construction of finite limits above show that limits in $\Asm(\cat B,(A,\cdot),\phi)$ lie direct above limits in $\cat B$. Hence $\uo$ preserves finite limits too.

Concerning regular epimorphisms: let $f,g:X\to Y$ be a parallel pair of $\phi$-total arrows which have a coequalizer $e:\uo Y\to Z$ in $\cat B$. Since $e$ is a regular epimorphism, $e_*(Z)$ is an assembly. Suppose $h:Y\to W$ satisfies $h\circ f = h\circ g$. Then there is a $k: Z\to \uo W$ such that $h = k\circ e$ and $k:e_*(Z) \to W$ is $\phi$-total because $h$ is.

If $f:X\to Y$ is any morphism in $\cat B$ then $\subob{y\of Y|\exists x\of X.f(x)=y}$ is the coequalizer of its kernel pair. Since every kernel pair in $\Asm(\cat B,(A,\cdot),\phi)$ lies above a kernel pair in $\cat B$, $\cat B$ has coequalizers of kernel pairs too.

This time $\uo$ preserves regular epimorphisms because it is left adjoint to $\nabla$. For each regular epimorphism $e:X\to Y$ in $\cat B$, $e_*(\nabla X) = \nabla Y$ and therefore $\nabla$ also preserves regular epimorphisms.

Regular epimorphisms are stable under pullback in $\Asm(\cat B,(A,\cdot),\phi)$ for the following reasons. Suppose $e:X\to Z$ and $f:Y\to Z$ are $\phi$-total morphisms and $e$ is regularly epic. Let $e':W\to \uo Y$ and $f':W\to \uo X$ be pullbacks of $e$ and $f$ along each other. The pullback of $e$ along $f$ in $\Asm(\cat B,(A,\cdot),\phi)$ then is $e':(e')^*(X)\otimes_W (f')(Y) \to Y$. The following shows that this $e'$ is a regular epimorphism because $e'_*((e')^*(Y)\otimes_W (f')^*(X))$ is isomorphic to $Y$. One can verify the following equations by raw computation
\begin{align*}
e'_*((e')^*(Y)\otimes_W (f')^*(X)) &= Y\otimes_{\uo Y} (e')_*(f')^*(X) \\
& = Y\otimes_{\uo Y} f^*(e_*(X))
\end{align*}
The $\phi$-total morphism $\id_{\uo Z}: e_*(X)\to Z$ is an isomorphism because $e:X\to Z$ is a coequalizer for the same kernel pair as $e:X\to e_*(X)$. The operator $f^*$ preserves this isomorphism, so $f^*(e_*(X))$ and $f^*(Z)$ are isomorphic. The morphism $\id_{\uo Y}: Y \to f^*(Z)$ is $\phi$-total because $f\circ\id_Y:Y\to Z$ is. Therefore $Y\otimes_{\uo Y} f^*(Z)$ is isomorphic to $Y$. Because $e'_*((e')^*(Y)\otimes_W (f')^*(X))$ is isomorphic to $Y$, $e'$ is a regular epimorphism. By generalization all regular epimorphism are stable under pullback.
\end{proof}

\subsection{Soundness}
This subsection holds the proof that the category of assemblies is a Heyting category. This means that intuitionistic logic is sound for realizability interpretations which are connected to categories of assemblies. To complete the proof that $\Asm(\cat B,(A,\cdot),\phi)$ is a Heyting category, I first demonstrate the equivalence of the subobject lattices in categories of assemblies with simpler preordered sets.

\newcommand\rleq{\mid\joinrel\stackrel{\mathbf r}{=}}
\begin{defin}
For each object $X$ of $\cat B$ and $U,V\in \sub(A\times X)$, let 
\[ h_X(U,V) = \subob{a\of R| \forall (b,x)\of U. a\cdot b\converges, (a\cdot b,x)\of V}\]
Let $U\rleq V$ if $h_X(U,V)\in \phi$. 
\end{defin}

Note that $\rleq$ is a family of binary relations.
The following establishes a family of equivalences between initial segments of the posets $(\sub(A\times X,\rleq)$ and subobjects posets of $\Asm(\cat B,(A,\cdot),\phi)$.

\hide{Maak een inverse functor. Dat is alles wat we hoeven te doen. }

\begin{defin} Subobjects are isomorphism classes of monomorphisms by definition \ref{heytcat}. For each subobject $U$ let $E_X:\sub(X) \to \sub(A\times \uo X)$ satisfy
\[ E_X(U) = \im{(\id_A\times \mu_U)}(\rr {X_U})\in \sub(A\times\uo X) \]
Here $\mu_U:X_U\to X$ is the monomorphism in $U$ chosen in definition \ref{promonos}.
\end{defin}

\begin{lemma} The map $E_X$ determines an equivalence between $\sub(X)$ and the initial segment $[\emptyset,\rr X]$ of $(\sub(A\times\uo X),\rleq)$. \label{insegequiv} \end{lemma}


\begin{proof}
Suppose that $U, V\in \sub(X)$ and $U\subseteq V$. There is a unique monic $p:X_U\to X_V$ such that $\mu_V\circ p = \mu_U$ by the definition of $\sub$. If $R\in \phi$ tracks $p$, then $R\subseteq h_X(E_X(U),E_X(V))$ and therefore $h_X(E_X(U),E_X(V))\in \phi$. Hence $E_X$ is a functor. By definition $E_X(X) = \rr X$ and since $X$ is terminal, $E_X:\sub(X)\to[\emptyset,\rr X]$.

Suppose $U,V\in \sub(X)$ satisfy $E_X(U)\rleq E_X(V)$. Now $h_X(E_X(U),E_X(V))$ is inhabited and therefore
\[ \cat B\models \forall y\of\uo X_U. \exists z\of\uo X_V. \mu_U(y)=\mu_V(z) \]
This means that the pullback $n:W\to X_U$ of $\mu_V$ along $\mu_U$ is a regular epimorphism. Because $\uo$ preserves finite limits, it preserves pullbacks and monomorphisms. Hence $n$ is an isomorphism. If $m$ is the pullback of $\mu_U$ along $\mu_V$ and $p = m\circ (n)^{-1}$ then $\mu_V\circ p = \mu_V$ as required. This proves that $E_X$ is full and faithful.
\[\xymatrix{
W \ar[r]^m \ar[d]^n \ar@{}[dr]|<\lrcorner & X_V \ar[d]^{\mu_V} \\
X_U \ar[r]_{\mu_U} \ar@/^/@{.>}[u]^{\pre n} & X
}\]

For each $\upsilon\rleq \rr X$ let $U = \subob{x\in \uo X|\exists a\of A.\upsilon(a,x)}$ and let $\mu_U: \uo X_U \to \uo X$ be the monomorphism that represents it (as in definition \ref{promonos}). The pair $(U,\pre{(\id_A\times \mu_U)}(\upsilon))$ is an assembly and $\mu_U:(U,\pre{(\id_A\times \mu_U)}(\upsilon))\to X$ is a $\phi$-total monomorphism. There is a $V\in \sub X$ such that $\mu_V$ is isomorphic to $\mu_U$ as monomorphism in $\Asm(\cat B,(A,\cdot),\phi))$. Hence $E_X$ is essentially surjective.
By the axiom of choice, $E_X$ has a weak inverse functor and is an equivalence of categories.
\end{proof}

\begin{lemma} For each object $X$ in $\Asm(\cat B,(A,\cdot),\phi)$, $\sub(X)$ is a Heyting algebra. \end{lemma} 

\begin{proof} Due to the equivalence $E_X$, $\sub(X)$ is a Heyting algebra if $[\emptyset,\rr X]$ is a biCartesian closed preordered set. I prove the latter statement.

It is easy to see that $[\emptyset, \rr X]$ has a top and bottom element.

Because $\Asm(\cat B,(A,\cdot),\phi)$ is regular, $\sub(X)$ has binary meets. I give a proof that binary meets exist in $[\emptyset, \rr X]$ because it helps to define Heyting implication later on. Remember that the interpretations of many \la-terms are members of $\phi$ by lemma \ref{computableterms}. One of those is the following \[\comb p = \lambda x,y,p.pxy\]
For each pair $U,V\in \sub(A\times \uo X)$ let:
\[ U\otimes V = \subob{((p\cdot a)\cdot b,x)\of A\times \uo X| p\of\db{\comb p},(a,x)\of U,(b,x)\of V }\]
This is a meet because for all $U,V,W\in \sub(A\times X)$:
\begin{align*}
\db{\lambda x.x(\lambda y,z.y)} &\subseteq h_X(U\otimes V, U) \\
\db{\lambda x.x(\lambda y,z.z)} &\subseteq h_X(U\otimes V, V) \\
\db{\lambda x.\comb p(h_X(U,W)x)(h_X(U,V)x)} &\subseteq h_X(U,V\otimes W)
\end{align*}
Hence $U\otimes V\rleq U$, $U\otimes V\rleq V$ and if $U\rleq V$ and $W$, then $U\rleq V\otimes W$.

The following \la-terms help to define binary joins.
\[ \comb i_0 = \lambda x,f,g. fx \quad \comb i_1 =\lambda x,f,g. gx \]
For each pair $U,V\in \sub(A\times \uo X)$ let:
\[ U\oplus V = \subob{(l\cdot a,x)\of A\times \uo X| l\of \db{\comb i_0}, (a,x)\of U}\cup\subob{(r\cdot a,x)\of A\times X| r\of \db{\comb i_1}, (a,x)\of V}\]
This is a join because for all $U,V,W\in \sub(A\times X)$:
\begin{align*}
\db{\comb i_0} &\subseteq h_X(U, U\oplus V)\\
\db{\comb i_1} &\subseteq h_X(U, U\oplus V)\\
\db{\lambda x.xh_X(U,W)h_X(V,W)} &\subseteq h_X(U\oplus V,W)
\end{align*}
Hence $U\rleq U\oplus V$ and $V\rleq U\oplus V$ and if $U\rleq W$ and $V\rleq W$, then $U\oplus V\rleq W$.

In order to get Heyting implication for each pair $U,V\in \sub(A\times \uo X)$ let:
\[ d(U,V) = \subob{(a,x)\of A\times \uo X| \forall (b,x)\of U.a\cdot b\converges \land (a\cdot b,x)\of V} \]
This is a Heyting implication because for all $U,V,W\in \sub(A\times X)$
\begin{align*}
\db{\lambda x.xh_X(U,d(V,W))} &\subseteq h_X(U\otimes V,W)\\
\db{\lambda x,y. h_X(U\otimes V,W)((\comb px)y)} &\subseteq h_X(U,d(V,W))
\end{align*}
Hence $U\rleq d(V,W)$ if and only if $U\otimes V\rleq W$.

All the required structure is present in $[\emptyset,\rr X]$. Therefore $\sub(X)$ is a Heyting algebra.
\end{proof}

The next lemma says that the inverse image map of any $\phi$-total morphism $f:X\to Y$ has a nice representation as monotone map $[\emptyset, \rr Y]\to [\emptyset, \rr X]$.

\begin{lemma} For each morphism $f:X\to Y$ of $\Asm(\cat B,(A,\cdot),\phi)$ there is a natural isomorphism between the functors $E_X\circ \pre f$ and $\pre{(\id_A\times f)}\circ E_Y$ going from $\sub(Y)$ to $[\emptyset, \rr X]$.\end{lemma}

\begin{proof} Naturalness is the easy part of this lemma, because $[\emptyset, \rr X]$ is a preordered set and every square in it is commutative. That leaves the problem of finding the isomorphisms.

For each $U\in \sub (Y)$, let $\mu_U:Y_U \to Y$ be its representation. 
The following equations hold by definition.
\begin{align*}
E_X\circ \pre f(U) &= \im{(\id_A\times \mu_{\pre f(U)})}(\rr{ X_{\pre f(U)}}) \\
\pre{(\id_A\times f)}\circ E_Y(U) &= \pre{(\id_A\times f)}\circ\im{(\id_A\times \mu_U)}(\rr{Y_U})
\end{align*}

There is a unique morphism $g: X_{\pre f(U)} \to Y_U$ such that the morphism $f$, $\mu_{\pre f(U)}$, $\mu_U$ and $g$ form a pullback square.
\[ \xymatrix{
X_{\pre f(U)} \ar@{.>}[r]^g\ar[d]_{\mu_{\pre f(U)}} \ar@{}[dr]|<\lrcorner & Y_U \ar[d]^{\mu_U} \\
X \ar[r]_f & Y
}\]
The universal property of the $g^*$ construction (see definition \ref{hulpfuncties}) makes $g^*(Y_U)$ another assembly over $\uo Y_U$ for which this square is a pullback. Hence $g^*(Y_U) \simeq X_{\pre f(U)}$ and
\[ E_X\circ \pre f(U) = \im{(\id_A\times \mu_{\pre f(U)})}(\rr{X_{\pre f(U)}}) \simeq \im{(\id_A\times \mu_{\pre f(U)})}\circ \pre{(\id_A\times g)}(\rr{Y_U}) \]

The Beck-Chevalley condition in $\cat B$ implies:
\[ \im{(\id_A\times \mu_{\pre f(U)})}\circ \pre{(\id_A\times g)} = \pre{(\id_A\times f)} \circ \im{(\id_A\times \mu_U)} \]
Therefore
\begin{align*}
E_X\circ \pre f(U) &\simeq \im{(\id_A\times \mu_{\pre f(U)})}\circ \pre{(\id_A\times g)}(\rr{Y_U})\\
& = \pre{(\id_A\times f)} \circ \im{(\id_A\times \mu_U)}(\rr{Y_U}) \\
& = \pre{(\id_A\times f)}\circ E_Y(U)
\end{align*}

Since this construction works for arbitrary subobjects, the inverse image map $\pre f$ in $\Asm(\cat B,(A,\cdot),\phi)$ commutes up to isomorphism with $\pre{(\id_A\times f)}$ in $\cat B$.
\end{proof}

The following lemma completes the proof that $\Asm(\cat B,(A,\cdot),\phi)$ is a Heyting category.

\begin{lemma} Let $f:X\to Y$ be a morphism of $\Asm(\cat B,(A,\cdot),\phi)$. Its inverse image map $\pre f$ is a morphism of Heyting algebras, it has both adjoints $\im f\dashv \pre f\dashv \forall_f$ and these adjoints satisfy the Beck-Chevalley conditions. \end{lemma}

\begin{proof} If $\pre f$ preserves Heyting implication and has a right adjoint, then the rest of the properties hold for the following reasons. In every regular category, $\pre f$ has a left adjoint which satisfies the Beck-Chevalley condition. That the left adjoint satisfies the Beck-Chevalley condition implies that the right adjoint does too. A functor that has both adjoints preserves all limits including all finitary meets and joins of the Heyting algebra. The rest of this proof shows that $\pre{(\id_A\times f)}:[\emptyset, \rr Y] \to [\emptyset,\rr X]$ has a right adjoint and preserves implication, so that the inverse image map does too.

There already is an adjunction between $\pre{(\id_A\times f)}$ and $\forall_{\id_A\times f}$ relative to the inclusion orders $\subseteq$ on $\sub(A\times \uo X)$ and $\sub(A\times \uo Y)$. Because $\alpha\subseteq\beta$ implies $\alpha\rleq \beta$ on either side; because $\pre{(\id_A\times f)}$ preserves $\rleq$, this adjunction is preserved.

Write out the definitions in order to see that 
\[ \pre{(\id_A\times f)}(d(\alpha,\beta)) = d(\pre{(\id_A\times f)}(\alpha),\pre{(\id_A\times f)}(\beta)) \]
\end{proof}

The main theory of this subsection is a straightforward corollary of the lemmas above:

\begin{theorem} For each Heyting category $\cat B$, each partial applicative structure $(A,\cdot)$ of $\cat B$ and each combinatory complete filter $\phi$, the category of assemblies $\Asm(\cat B,(A,\cdot), \phi)$ is a Heyting category. \label{Heyting} \end{theorem}

\begin{remark}[Realizability toposes]\newcommand\exreg{\textrm{ex/reg}}
By the way, \emph{realizability toposes} are \emph{ex/reg completions} of specific categories of assemblies.
The ex/reg completion $\cat C_\exreg$ of a regular category $\cat C$ freely adds quotients of internal equivalence relations, while preserving regular epimorphism in $\cat C$. The construction is a left biadjoint to the inclusion of regular categories and regular functor into the category of exact categories and regular functors \cite{MR1600009, MR1358759, MR1056382}. I have studied the properties of ex/reg completions of categories of assemblies in \cite{RealCats, MSC:8896618}. The category $\Asm(\cat B,(A,\cdot),\phi)_\exreg$ is a topos if $\cat B$ is, because in that case $\Asm(\cat B,(A,\cdot),\phi)$ has a \emph{generic monomorphism} (see \cite{Menni00exactcompletions, MR1900904}).
\end{remark}

\subsection{Realizability}\label{real}
The category of assemblies are a tool for studying many forms of realizability. This section shows how to connect a realizability interpretation to a category of assemblies.

As in ordinary categorical logic, the internal language of a category of assemblies assigns subobjects of its terminal object to each proposition. The lattice of subobjects of $1$ in $\Asm(\cat B,(A,\cdot), \phi)$ is equivalent to the poset $(\sub(A), \rleq)$ in $\cat B$ by lemma \ref{insegequiv}. Thus the category of assemblies assigns an equivalence class of subobjects of $A$ to each proposition. Realizability makes these choices inductive on subformulas.

\newcommand\rlss{\mathrel{\bf r}}
\begin{defin}[Realizability] I recursively define the formula $a\rlss p$, where $p$ is a formula of the internal language of $\Asm(\cat B,(A,\cdot),\phi)$ and $a$ is a variable symbol of type $A$. This is called the realizability relation.

For readability, let 
\[ \comb 0 = \db{\lambda x,y.x}, \comb 1 = \db{\lambda x,y.y}, A_2 = \subob{a\of A| \forall k\of\comb 0\cup\comb 1. a\cdot k\converges} \]

The following clauses define the realizability relation.
\begin{align*}
(a\rlss t=s) \iff & t=s \\
(a\rlss \top) \iff & \top \\
(a\rlss \bot) \iff & \bot \\
(a\rlss p\land q) \iff & a\of A_2\land \forall k\of \comb 0. (a\cdot k\rlss p) \land \forall l\of \comb 1. (a\cdot l\rlss q) \\
(a\rlss p\vee q) \iff & (a\of A_2 \land (\forall k\of\comb 0. a\cdot k\of \comb 0)\land (\forall l\of\comb 1. (a\cdot l\rlss p)))\vee \\
& (a\of A_2 \land (\forall k\of\comb 0. a\cdot k\of \comb 0)\land (\forall l\of\comb 1. (a\cdot l\rlss p)))\\
(a\rlss p\to q) \iff & \forall b\of A. (b\rlss p) \to a\cdot b\converges\land (a\cdot b \rlss q)\\
(a\rlss \exists x\of X. p) \iff & a\of A_2 \land \exists x\of \uo X.\forall k\of \comb 0. (a\cdot k,x)\of \rr X \land  \forall l\of \comb 1. (a\cdot l \rlss p) \\
(a\rlss \forall x\of X. p) \iff & \forall b\of A,x\of \uo X. (b,x)\of \rr X \to  a\cdot b\converges\land (a\cdot b \rlss p)
\end{align*}
The realizability of atomic relations, which come form subobjects $U\subseteq X$ in $\Asm(\cat B, (A,\cdot), \phi)$ is determined by a choice between the representations that $U$ has in $[\emptyset,\rr X]$. For the chosen $V\subseteq A\times \uo X$ realizability is defined as follows:
\[ a\rlss x\of U \iff (a,x)\of V \]

The realizability interpretation \emph{satisfies} a proposition $p$ if $\subob{a\of A| a\rlss p}\in \phi$.
\end{defin}

\begin{remark}[Diversity]
The definitions of $\forall$, $\to$ and validity impose a preorder on formulas, which forms a Heyting algebra in sound realizability interpretations. This determines how the rest of the logic is realized up to equivalence. So even though realizability interpretations in the literature may differ in the details, they often satisfy the same propositions.
\end{remark}

\begin{lemma} Let $p$ be a proposition in the internal language of $\Asm(\cat B,(A,\cdot), \phi)$ and let $\tau_p\in \sub(1)$ be its truth value. Then $\subob{a\of A| a\rlss p}$ is equivalent to $E_1(\tau_p)$. \end{lemma}

\begin{proof} Left as exercise. \end{proof}

\begin{theorem} Let $p$ be a proposition in the internal language of $\Asm(\cat B,(A,\cdot), \phi)$ and let $\tau_p\in \sub(1)$ be its truth value. Then $\tau_p = \top$ if and only if the realizability interpretation satisfies $p$. \label{abracadabra} \end{theorem}

\begin{proof} The equivalence of $E_1(\top)$ and $\subob{a\of A| a\rlss p}$ means that one is in $\phi$ if the other is, but $E_1(\top)=A\in \phi$. \end{proof}

Categories of assemblies are connected to an internal form of realizability in arbitrary Heyting categories. The filters add a lot of flexibility.
\begin{itemize}
\item Ordinary realizability interpretations satisfy every formula that has realizers. This corresponds to the filter of inhabited subobjects of a partial combinatory algebra.
\item The correct interpretation of `having realizers' may be that the set of realizers has a global section, however. In that case the filter set of subobjects of $A$ which have a global section determine the category of assemblies.
\item In relative realizability, $A$ has a subset $A'$ of special realizers and a proposition is valid if it has a realizer in $A'$. Having realizers can refer to either inhabited subobjects or to global sections. In the first case $A'$ is a subobject of $A$ and the filter is the set of subobject that intersect $A'$. In the second case $A'$ is a set of global section and the filter is the set of subobjects through which some of these globals sections factorize.
\item Filters are closed under intersections and hence realizability interpretations are too. This is a new construction for realizability models as far as I know.
\end{itemize}
Thus categories of assemblies cover a lot of ground as far as realizability in concerned.

\section{Realizability categories}\label{Realcats}
This section contains a characterization of the categories that are equivalent to a category of assemblies. Later sections explain to what extend and how the characteristic properties of these categories can be expressed in their internal language.

\subsection{Characteristic properties}\newcommand\A{{\mathring A}}
The relevant properties of realizability catego\-ries are complicated enough to devote a subsection to their definition.

The definition involves some extra structure of $\Asm(\cat B, (A,\cdot), \phi)$ for the following reason. In any Heyting category $\cat B$ the terminal object $1$ and the unique map $!:1\times 1 \to 1$ form a partial applicative object and $\set 1$ is combinatory complete filter. In this case $\uo:  \Asm(\cat B, (1,!), \set 1) \to \cat B$ is an equivalence of categories. This means that every Heyting category is a realizability category in a trivial way.

The following extra structure is taken in consideration. 
The category of assemblies $\Asm(\cat B,(A,\cdot), \phi)$ extends the category $\cat B$ with new subobjects. There is a special new subobject $\A$ (which is introduced in definition \ref{ringa}) which generates all others in some sense. The definition of realizability category characterizes the combination of the inclusion $\nabla:\cat B \to \Asm(\cat B, (A,\cdot), \phi)$, the underlying object functor $\uo$ and this assembly $\A$.

Throughout this section, let $\cat B$ and $\cat C$ be Heyting categories let $F:\cat B \to \cat C$ and $U:\cat C\to\cat B$ be two functors, and let $C$ be some object of $\cat C$.

\begin{axiom}[Separability] The functor $F$ is right adjoint to $U$. Both functors are regular. The unit $\eta:\id_{\cat C} \to FU$ is a natural monomorphism. The co-unit $\epsilon: UF \to \id_{\cat C}$ is a natural isomorphism. \label{separ}
\end{axiom}

\begin{remark} I call this the separability axiom because $U$ is a kind of fibred category. To be precise The axiom forces $U$ to be a Street fibration, i.e.\ the composition of a Grothendieck fibration and an equivalence of categories. More relevantly, the axiom also tells us that the Grothendieck part is \emph{separated} relative to the regular topology on $\cat B$. \end{remark}

The following definition helps to formulate the next axiom, which says that every object is an assembly.

\begin{defin}[Prone]\label{prone1} An arrow $p: X\to Y$ is \emph{prone} (or \emph{Cartesian}) if for each $f:Z\to Y$ such that $Uf = Up\circ g$ for some $g:UZ\to UX$, there is a unique $h:Z\to X$ such that $p\circ h = f$ and $Uh = g$.
\end{defin}

The following characterization of prones is very useful in this context.

\begin{lemma}\label{prone2} A morphism $f:X\to Y$ of $\cat C$ is prone if and only if the naturalness square $FUf\circ \eta_X = \eta_Y\circ f$ is a pullback.
\[ \xymatrix{
X\ar[r]^f\ar[d]_{\eta_X} \ar@{}[dr]|<\lrcorner & Y \ar[d]^{\eta_Y}\\
FUX \ar[r]_{FUf} & FUY
}\]
\end{lemma}

\begin{proof} Left as exercise. \end{proof}

\begin{axiom}[Weak genericity] For each object $X$ of $\cat C$ there is a span $p: Y\to C$, $e:Y\to X$ where $p$ is prone and $e$ is regularly epic. \label{weakgen} \end{axiom}

\begin{remark} For an ordinary fibred category $U:\cat C\to\cat B$ a \emph{generic object} is an object $G$ of $\cat C$ such that there is a prone morphism $X\to G$ for every object $X$ of $\cat C$. The object $G$ the axiom only generates a cover for each object of $\cat C$, i.e.\ it is only \emph{weakly generic}. This is why I call this axiom the weak genericity axiom. \end{remark}

The last axiom says that every morphism is $\phi$-total for the filter of $R\subseteq UC$ for which $FR$ intersects $C$.

\begin{axiom}[Tracking] There is a partial operator $\cdot: C\times C\partar C$ for which the inclusion $\dom(\cdot)\subseteq C\times C$ is prone and which has the following property. Let $f = (f_0,f_1): Y \to C\times X$ be a monic such that $f_1: Y\to X$ is a regular epimorphism, let $p:P\to C$ be prone and let $g:P\to X$. There is an inhabited $R\subseteq C$ such that $r\cdot p(x)\converges$ for all $x\of P$ and $(r,x)\mapsto (r\cdot p(x),g(x))$ factors though $f$.
\[ \xymatrix{
\dom(\cdot) \ar[r]_{\pi_1}\ar@/^/[rr]^\cdot & C & C & Y \ar[l]_{f_0} \ar[d]^{f_1} \\
R\times P \ar[u]^{\id_R\times p} \ar[r]_{\pi_1}\ar@{.>}@/_/[urrr] & P\ar[u]_p \ar[rr]_g  && C
}\]\label{tracking}
\end{axiom}

The upcoming subsection explains why $\Asm(\cat B,(A,\cdot), \phi)$ is a realizability category. The one after that shows that for every category that satisfies these axioms there is an equivalent category of assemblies.

\hide{ Bedenk wat een zwak generieke deelcategorie allemaal kan doen. Je hebt de relevante noties van modest set en van partitioned assembly. Beide tracking axioma's zijn waar, en het laatste axioma kunnen we vervangen door een zwakke vorm van Cartesische geslotenheid.}

\hide{ beperkte exponentiatie: $C^X$ bestaat voor alle objecten. Zelfs $C_\bot^X$. Deze representeert partiele functies met gesloten domain. Vreemd eigenlijk. Het klopt ook niet wegens extensionele equivalentie. Antwoord: met zwakke exponenti\"ele objecten schieten we niks op. }

\subsection{Satisfaction}
This subsection demonstrates that $\Asm(\cat B,(A,\cdot),\phi)$ satisfies the axioms in the previous subsection for any partial applicative structure $(A,\cdot)$ in $\cat B$ and any combinatory complete filter $\phi$. Theorem \ref{basis} says that axiom \ref{separ} holds if $F = \nabla$ and $U = \uo$. 
I now formally introduce the assembly $\A$ which takes the place of $C$ in each category of assemblies.

\begin{defin} Let $(A,\cdot)$ be a partial applicative structure and let $\phi$ be a combinatory complete filter. Let $\A$ be the diagonal assembly $(A,\subob{(a,a)\of A\times A| a\of A})$. \label{ringa}\end{defin}

The next theorem justifies axiom \ref{weakgen}.

\begin{theorem}[Weak genericity] For each assembly $X$ of $\Asm(\cat B, (A,\cdot),\phi)$ there is an assembly $Y$ with a prone morphism $p:Y\to \A$ and a regular epimorphism $e:Y\to X$.\end{theorem}

\begin{proof} For each object $X$ of $\Asm(\cat B, (A,\cdot),\phi)$, any monomorphism $(p,e):Y\to A\times \uo X$ that represents $\rr X$, $X = e_*(p^*(\A))$ by definition (see definition \ref{hulpfuncties} for $e_*$ and $p^*$). Let $Y = p^*(\A)$, then $p:Y\to X$ is prone and $e:Y\to X$ is a regular epimorphism. \end{proof}

The following lemma explains the role of the filter $\phi$.

\begin{lemma} For all $I\in \sub A$, $\nabla I\cap \A$ is inhabited if and only if $I\in \phi$.\label{inhab} \end{lemma}

\begin{proof} If $\nabla I\cap \A$ is inhabited, then so is $I$, because $\uo(\nabla I\cap \A) = I$ and $\uo$ preserves inhabited objects. Since $I\in \phi$ are also inhabited, assume that $I$ is inhabited and focus on the potentially non-total map $\id_1:\top \to !^*(\nabla I\cap \A)$ (see definition \ref{hulpfuncties} for $!^*$). If $I\in \phi$, then $\db{\lambda x,y.x}\cdot I\in \phi$ tracks $\id_1$ and $\nabla I\cap \A$ is inhabited. If on the other hand $\nabla I\cap \A$ is inhabited, then some $R\in \phi$ tracks $\id_1:\top \to !^*(\nabla I\cap \A)$. This means that $\im{(\cdot)}(R\times A)\subseteq I$ and that $I\in \phi$.
\end{proof}

Axiom \ref{tracking} follows.

\begin{theorem}[Tracking] There is a partial morphism $\triangledown: \A\times \A \partar \A$ for which the inclusion $\dom\triangledown\to \A\times \A$ is prone that has the following property. Let $f = (f_0,f_1): Y\to \A\times X$ be a monomorphism, such that $f_1$ is a regular epimorphism. Let $p:P\to \A$ be prone and let $g:P\to \A$ be an arbitrary morphism. There is an inhabited $R\subseteq \A$ such that $r\cdot p(x)\converges$ and $(r,x)\mapsto (r(x)\cdot p(x),g(x))$ factors through $Y$.
\end{theorem}

\begin{proof} The functor $\nabla$ turns application operator $\cdot: A\times A\partar A$ to an application operator $\nabla(\cdot): \nabla A\times \nabla A \partar \nabla A$. The operator $\triangledown$ is the restriction of $\nabla\cdot$ to $\nabla \dom(\cdot)\cap \A\times\A$. The inclusion of $\nabla \dom(\cdot)\cap \A\times\A$ into $\A\times \A$ is prone as required.

I use some of the operators of definition \ref{hulpfuncties} now. 
Because $p:P\to \A$ is prone, $p^*\A\simeq P$. Because $f_1:Y\to X$ is a regular epimorphism, $(f_1)_*(Y)\simeq X$. Since no generality is lost, assume $P=p^*\A$ and $X = (f_1)_*(Y)$. There is an $I\in \phi$ that tracks $g:p^*\A\to (f_1)_*Y$ and a $J\in \phi$ that tracks $f_0:Y\to C$. Let
\[ J\circ I = \subob{(b\cdot j)\cdot i\of A|b\of\db{\lambda x,y,z.(xy)z},j\of J,i\of I }\]
Under these circumstances for each $r\of I\circ J$ and $x\of \uo P$, there is unique $y\of \uo Y$ such that $g(x) = f_1(y)$ and $r\circ p(x) = f_0(y)$. 
Therefore there is a morphism $h:(J\circ I)\times \uo P \to Y$ such that $f_1\circ h = g\circ \pi_1$ and $f_0(h(r,x)) = r\cdot p(x)$. Let $R = \A\cap (J\otimes I)$. There is a $K\in \phi$ that tracks $g\circ\pi_1:R\times P\to X$. This $K$ also tracks $h:R\times P\to Y$, which is therefore a $\phi$-total morphism. The object $R$ is an inhabited subobject of $A$. Therefore the tracking axiom applies to categories of assemblies.
\end{proof}

The axioms of realizability categories are sound for categories of assemblies. The next subsection tackles completeness.

\subsection{Characterization theorem}\label{characterization}
The last subsection explained why categories of assemblies are realizability categories. This subsection explains why every realizability category is equivalent to a category of assemblies.

Throughout this subsection let Heyting categories $\cat B$, $\cat C$, functors $F:\cat B\to \cat C$, $U:\cat B\to \cat C$ and an object $C$ of $\cat C$ satisfy the axioms for a realizability category. Let $A = UC$, let $\cdot:A\times A\partar A$ equal $U(\cdot):UC\times UC \partar UC$ and let $\phi$ contain those subobjects $S$ of $A$ such that $FS$ intersects $C$. This subsection demonstrates that there is an equivalence $G:\Asm(\cat B,(A,\cdot),\phi)\to \cat C$ which satisfies $G\nabla \simeq F$ and $G\A \simeq C$. I first define an object map for the equivalence. Because $F$ is regular, by lemma \ref{reglog} $F$ induces a transformation $F:\sub(-) \to \sub(F-)$ which commutes with both $\pre f$ and $\im f$ for all morphisms $f$ of $\cat B$.

\begin{defin}\label{ksi} For each assembly $X$, let 
\[ \Xi(X) = \subob{x\of F\uo X| \exists c\of C. (c,x)\of F_{A\times\uo X}(\rr X)} \]
Let $GX$ be $F\uo X_{\Xi(X)}$ and let $g_X = \mu_{\Xi(X)}:GX \to F\uo X$ ($\mu$ from definition \ref{promonos}).
\end{defin}

\begin{lemma} For each $f:X\to Y$ in $\Asm(\cat B,(A,\cdot), \phi)$ there is a unique $h:GX\to GY$ such that $Ff\circ g_X = g_Y\circ h$.
\[\xymatrix{
GX\ar[d]_{g_X} \ar@{.>}[r]^h & GY \ar[d]^{g_Y} \\
F\uo X \ar[r]_{Ff} & F\uo Y
}\]\label{implicit functor}
\end{lemma}

\begin{proof} Since $g_Y$ is monic, if $g_Y\circ h'= g_Y\circ h$, then $h'= h$, hence uniqueness.

There is an $R\in \phi$ that tracks $f:X\to Y$. For each $(c,x)\of F(\rr X)$ and $r\of FR\cap C$, $r\cdot c\converges$ and $(r\cdot c,Ff(x))\of F(\rr Y)$ because $R$ tracks $f$ and $r\cdot c\of C$ because $C$ is closed under application. For this reason, the restriction of $Ff$ to $\Xi(X)$ factors through $\Xi(Y)$ and that means there is a map $h:GX \to GY$ such that $Ff\circ g_X = g_Y\circ h$.
\end{proof}

\begin{defin} For each $f:X\to Y$ in $\Asm(\cat B,(A,\cdot), \phi)$ let $Gf$ be the unique morphism $h:GX\to GY$ that satisfies $h\circ g_X = g_Y\circ h$. \end{defin}

\begin{lemma} The maps $G$ form a functor $\Asm(\cat B,(A,\cdot), \phi)\to \cat C$ which satisfies $G\nabla \simeq F$ and $G\A\simeq C$. \end{lemma}

\begin{proof} 
The arrow map $G$ preserves composition and identities for the following reason. For any chain of arrows $\set{f_i:X_i \to X_{i+1}|i< n}$, $g_{X_n}\circ (Gf_{n-1} \circ\dotsm\circ Gf_0) = F(f_{n-1}\circ\dotsm\circ f_0)\circ g_{X_0}$ by definition of $G$. By lemma \ref{implicit functor}, $Gf_{n-1} \circ\dotsm\circ Gf_0 = G(f_{n-1}\circ\dotsm\circ f_0)$ and therefore $G$ is a functor.

The functor $G$ now satisfies $G\nabla \simeq F$ for the following reason. For each $X$ in $\cat B$,
\[ \Xi(\nabla X) = \subob{x\of FX| \exists c\of C. (c,x)\of FA\times FX} \]
Therefore $g_X:G\nabla X\to FX$ is an isomorphism.

For $\A$,
\[ \Xi(\A) = \subob{a\of FA| \exists c\of C. c=a} \]
For this reason $G\A$ is isomorphic to $C$.
\end{proof}

Both categories $\cat C$ and $\Asm(\cat B,(A,\cdot), \phi)$ satisfy the weak genericity axiom, but in $\Asm(\cat B,(A,\cdot),\phi)$ each object $X$ has a prone-regularly-epic span $p:Y\to A$, $e:Y\to X$ where $(p,e):Y\to A\times X$ is monic. Proving that $G$ is an equivalence becomes simpler after proving that $\cat C$ has jointly monic prone-regularly-epic spans too.

\begin{lemma}[jointly monic spans] For each object $X$ of $\cat C$, there is a jointly monic prone-regularly-epic span $p:Y\to A$, $e:Y \to X$. \label{JMS} \end{lemma}

\begin{proof} By weak genericity, there is a prone-regularly-epic span $q:Z\to A$ and $f:Z\to X$. The morphism $(q,f): Z\to A\times X$ factors as a regular epimorphism $e':Z\to Y$ followed by a monomorphism $(p,e):Y\to A\times X$. The morphism $e: Y \to X$ is regularly epic because $e\circ e'= f$ and both $e'$ and $f$ are regularly epic. Morphism $p$ is prone for the following reasons.
For each regular epimorphism $r:P\to Q$ in $\cat C$ the naturalness square of $r$ and the unit $\eta: \id_\cat C \to FU$ is a pullback, because $FU$ preserves regular epimorphisms and $\eta$ is monic and because every square with vertical regular epimorphisms and horizontal monomorphisms is a pullback square.
\[\xymatrix{ P\ar[r]^r\ar[d]_{\eta_P}\ar@{}[dr]|<\lrcorner & Q \ar[d]^{\eta_Q} \\
FUX\ar[r]_{FUr} & FUQ
}\]
Therefore every regular epimorphism is a prone morphism. It is easy to see that prone morphisms satisfy 2-out-of-3 using lemma \ref{prone2}, so because $p\circ e' = q$ and both $e'$ and $q$ are prone, so is $p$. 
\end{proof}

\begin{lemma} The functor $G:\Asm(\cat B, (A,\cdot), \phi)\to\cat C$ is an equivalence of categories.\end{lemma}

\begin{proof}
The reasons that $G$ is essentially surjective on objects are the following.
Weak genericity says that for each object $Y$ of $\cat C$, there is a span consisting of a regular epimorphism $e:Z\to Y$ and a prone morphism $p:Z\to C$. Using one of these spans, define the assembly $HY$ as follows:
\[ HY = (UY, U\subob{ (a,y) \in C\times Y | \exists z\in Z. a=p(z), y=e(z) })\]

With $\Xi$ from definition \ref{ksi} regularity of $F$ and $U$ implies that:
\[ \Xi(HY) = \subob{x\of FUY| \exists c\of C, z\of FUZ. c=FUp(z), y=FUe(z) }\]

Because $p:Z\to C$ is prone, the naturalness square of the unit $\eta:\id_{\cat C} \to FU$ and $p$ is a pullback by lemma $\ref{prone2}$:
\[ \xymatrix{
Z\ar[r]^p\ar[d]_{\eta_Z} \ar@{}[dr]|<\lrcorner & C \ar[d]^{\eta_C} \\ 
FUZ \ar[r]_{FUp} & FA
}\]
So $\eta_{Z}:Z \to FUZ$ is an element of the subobject $\subob{ z\of FUZ| c=FUp(z) }$. The subobject $\Xi(HY)$ is just $\exists_{FU e}(Z)$, but this subobject contains the monic $\eta_Y: Y\to FUY$, because it is the monic part of an regularly-epic-monic factorization of $FUe\circ \eta_Z$ by naturalness of $\eta$.
\[\xymatrix{
Z \ar[r]^e\ar[d]_{\eta_Z} & Y\ar[d]^{\eta_Y} \\
FUZ\ar[r]_{FUe} & FUY
}\]
Both $\eta_Y:Y\to FUY$ and $g_{GH}:GHY \to FUY$ are members of $\Xi[HY]$, so they are isomorphic. Hence $G$ is essentially surjective on objects. 

Concerning faithfulness: by definition $UC = A$. Since $U$ is a regular functor,
\[ U\Xi[X] = \subob{x\of UF\uo X| \exists c\of A. (c,x)\of U_{F(A\times \uo X)}F_{A\times \uo X}(\rr X)} \]
The isomorphism $\epsilon:UF\to \id_{\cat B}$ together with the definition of $A$-assembly imply that the transpose $g^t:UG\to \uo$ is a natural isomorphism. Since $\uo$ is faithful, so is $G$.

Fullness of $G$ relies on tracking. For each assembly $X$ let the monomorphism $(p_X, e_X):R_X \to C\times GX$ represent the following subobject:
\[ \subob{(c,x)\of C\times GX | (\eta_C(c),g_X(x))\of F_{A\times\uo X}(\rr X) } \]
This way $p_X: R_X\to C$ and $e_X:R_X \to GX$ is a jointly monic prone-regularly-epic span for $GX$. 

Let $Y$ be another assembly and let $h:GX\to GY$ be any map.
These morphisms satisfy the conditions of the tracking axiom: $(p_Y, e_Y): R_Y\to C\times GY$ is a monomorphism, $e_Y:R_Y\to GY$ is an epimorphism, $p_X:R_X\to C$ is prone and there is a map $h\circ e_X:R_X\to GY$. So there is an inhabited $S\subseteq C$ such that $s\cdot p_X(r)\converges$ for all $r\of R_X$ and $s\of S$ and such that there is a morphism $t: S\times R_X\to C\times R_Y$ that satisfies $t_0 = (\cdot)\circ (\id_S\times p_X)$ and $t_1 = h\circ e_X\circ \pi_1$ factors through $(p_Y,e_Y)$.
\[ \xymatrix{
\dom(\cdot) \ar[r]_{\pi_1} \ar@/^/[rr]^\cdot & C & C \\
S\times R_X \ar[u]^{\id_R\times p_X}\ar[r]_{\pi_1}\ar@{.>}@/^/[rr]^(.4)t & R_X\ar[u]^{p_X} \ar[d]_{e_X} & R_Y\ar[d]^{e_Y}\ar[u]_{p_Y}\\
& GX \ar[r]_h & GY
}\]

Because the inclusion of the domain of the application operator is prone, $(s\cdot r) \mapsto s\cdot p(r)$ defines a map $C\cap FUS \times R_X\to C$. Because $p_Y:R_Y\to C$ is prone too, there is a unique extension of $t$ to $C\cap FUS\times R_X$. The object $US$ belongs to $\phi$ because $S\subseteq C\cap FUS$ and $S$ is inhabited.

The transpose $g^t:UG\to \uo$ is an isomorphism. Hence there is a unique $h':\uo X\to \uo Y$ commuting with $Uh: UGX\to UGY$. This $h'$ is $\phi$-total since $US\in \phi$. Finally, $Gh' = h$ by definition of $G$.
\end{proof}

By generalization the content of this subsection summarizes as the following theorem.

\begin{theorem}[Characterization] Let $\cat B$ and $\cat C$ be Heyting categories, let $F:\cat B\to\cat C$ and $U:\cat C \to \cat B$ be two functors and let $C$ be some object of $\cat C$. If these satisfy axioms of realizability categories, then $\cat C$ is equivalent to $\Asm(\cat B,(A,\cdot),\phi)$ for some partial applicative structure $(A,\cdot)$ and some combinatory complete filter $\phi$.\label{charac}
\end{theorem}

The axioms of realizability categories indeed determine if an embedding of a Heyting categories $F:\cat B\to \cat C$ is equivalent to $\nabla:\cat B \to \Asm(\cat B, (A,\cdot),\phi)$ for some partial applicative structure $(A,\cdot)$ and some combinatory complete external filter $\phi$. In that sense the axiomatization is complete.

\subsection{Slices} As conclusion to this paper, I show one application of the axioms of realizability: the proof that slices of categories of assemblies are realizability categories.

\begin{defin} For each assembly $I$ in $\Asm(\cat B, (A,\cdot), \phi)$ let $\cat B/I$ be the category of prone maps into $I$. Let $A/I$ be the projection $\nabla A\times I\to I$ and let $(A/I,\cdot)$ the constant partial applicative structure. Finally, for each $U\subseteq \nabla A\times I$ let $\pi_1:U\to I$ be in $\phi/I$ if the projection is a regular epimorphism.
\end{defin}

\begin{lemma} The categories $\Asm(\cat B, (A,\cdot), \phi)/I$ and $\Asm(\cat B/I, (A/I,\cdot), \phi/I)$ are equivalent. \end{lemma}

\begin{proof} The category $\cat B/I$ is a reflective subcategory of $\Asm(\cat B, (A,\cdot), \phi)/I$. The reflector sends $x:X\to I$ to the pullback $Tx$ of $\nabla\uo x$ along $\eta_I: I\to \nabla\uo I$ and since both reindexing along $\eta_I$ and $\nabla \uo$ are regular functors, so is $T$. The unit $\theta_x$ of the reflector is the factorization of $\eta_X:X \to \nabla\uo X$ through the projection $TX = \nabla\uo X\times_{\nabla\uo I} I\to \nabla\uo X$, which is a monomorphism, because it is the pullback of the monomorphism $\eta_I$ and equal to $\eta_{TX}$. Hence $\theta_x$ is a monomorphism and the reflector $T$ is faithful.
\[\xymatrix{
X\ar@/^2ex/[rr]^{\eta_X}\ar[r]_{\theta_x}\ar[d]_x & TX \ar[r]_{\eta_{TX}} \ar[d]_{Tx}\ar@{}[dr]|<\lrcorner & FUX \ar[d]^{FUx} \\
I\ar@{=}[r] & I \ar[r]_{\eta_I} & FUI
}\]
So the slice category satisfies separability.

Concerning weak genericity. For each $x:X\to I$ there is an object $Y$, a prone morphism $p: Y\to \A$ and a regular epimorphism $e:Y\to \A$ by the weak genericity axiom. The morphism $(p,e\circ x): Y\to \A\times I$ is prone relative to the reflector $T:\Asm(\cat B, (A,\cdot), \phi)/I \to \cat B/I$ for the following reasons.
Let $g = (g_0, g_1): Z\to \A\times I$ satisfy $Tg = T(p, e\circ x)\circ h$ for some $h: TZ \to TY$. 
For each $x:X\to I$ the morphism $\nabla \uo \theta_x$ is an isomorphism by separability in $\Asm(\cat B,(A,\cdot),\phi)$. So let
\[ h' = (\nabla\uo\theta_{x\circ e})^{-1}\circ \nabla\uo h \circ \nabla\uo\theta_{g_1} \]
By naturalness of $\theta$,  $\nabla\uo g_0 = \nabla\uo  p\circ h'$. Since $p$ is prone relative to $\nabla\uo$ there is a unique $k: Z \to Y$ such that $\nabla \uo k=h'$ and $p\circ k = g_0$. Because $\nabla\uo\theta$ is a natural isomorphism, $\nabla\uo Tk = \nabla\uo h$ and $Tk = h$ by faithfulness of $\nabla\uo$.
\[ \xymatrix{
\nabla\uo Z \ar[rr]^{h'} \ar[d]_{\nabla\uo\theta_{g_0}} && \nabla\uo Y\ar[d]^{\nabla \uo\theta_{x\circ e}} \ar[rr]^{\nabla\uo(p,x\circ e)} && \nabla \uo( C\times I ) \ar[d]^{\nabla\uo\theta_{\pi_1}}\\
\nabla\uo TZ \ar[rr]_{\nabla\uo Tk=\nabla\uo h} && \nabla\uo TZ \ar[rr]_{\nabla\uo T(p,x\circ e)} && \nabla\uo( T(C\times I) )
}\]
By assumption $Tg = T(p, e\circ x)\circ h$ and $Tg_1 = T(e\circ x)\circ h = T(e\circ x\circ k)$ because $T$ preserves limits. The functor $T$ is faithful, so $g_1 = e\circ x\circ k$. Hence $g = (p, x\circ e)\circ k$. By generalization $(p,x\circ e)$ is prone relative to $T$ meaning there is a prone-regularly-epic span for each object in $\Asm(\cat B, (A,\cdot), \phi)/I$.

The stability of the tracking axiom follows from the stability of the kinds of arrows involved. A morphism $p: P \to C\times I$ is prone in $\Asm(\cat B,(A,\cdot),\phi)/I$ if it is the pullback of $Tp: TP \to FUC\times I$ along $\eta_C\times I: C\times I \to FUC\times I$ by lemma \ref{prone2}. Because the projection $\pi_0:C\times I \to C$ is the pullback of the projection $\pi_0:FUC\times I \to FUC$ along $\eta_C: C\to FUC$, the composition $\pi_0\circ p: TP \to C$ is prone in $\Asm(\cat B,(A,\cdot),\phi)$.
\[ \xymatrix{
P \ar[r]^p\ar[d]_{\theta_P} \ar@{}[dr]|<\lrcorner & C\times I \ar[r]^{\pi_0}\ar[d]^{\eta_C\times I} \ar@{}[dr]|<\lrcorner & C\ar[d]^{\eta_C} \\
TP \ar[r]_{Tp} & FUC\times I\ar[r]_{\pi_0} & FUC
}\]
Let $f = (f_0,f_1): Y\to (C\times I)\times_I X$ be a monomorphism such that $f_1: Y\to X$ is a regular epimorphism. Of course $(C\times I)\times_I X \simeq C\times X$, as $C\times I$ stands for a constant object of the slice. Any morphism $P\to X$ factors though $f_1$ because of the tracking axiom in $\Asm(\cat B,(A,\cdot),\phi)$ and this factorization works in the slice $\Asm(\cat B,(A,\cdot),\phi)/I$ because $C\times I$ is a constant object.

The conclusion is that $\Asm(\cat B,(A,\cdot),\phi)/I$ is a realizability category and that it is isomorphic to $\Asm(\cat B/I, (A/I,\cdot), \phi/I)$ in particular.
\end{proof}

The conclusion of this paper generalizes the lemma proved above.

\begin{theorem} Slices of categories of assemblies are categories of assemblies. \end{theorem} 

\subsection*{Acknowledgments} 
I am grateful to the Warsaw Center of Mathematics and Computer Science for the opportunity to write this paper. The contents of this paper is mostly based in research during my Ph.D. candidacy at the Mathematical Institute of Utrecht University.

\bibliographystyle{plain}
\bibliography{realizability}{}


\end{document}